\documentclass[12pt]{amsart}
\usepackage[all]{xy}
\usepackage[latin1]{inputenc}
\usepackage{amssymb}
\newcommand{\C}{{\mathbb C} }
\newcommand{\R}{{\mathbb R} }

\newcommand{\cD}{{\mathcal D} }
\newcommand{\cE}{{\mathcal E} }

\newcommand{\cL}{{\mathcal L} }
\newcommand{\cM}{{\mathcal M} }
\newcommand{\cN}{{\mathcal N} }
\newcommand{\cO}{{\mathcal O} }

\newcommand{\cT}{{\mathcal T} }

\newcommand{\cX}{{\mathcal X} }

\newcommand{\cZ}{{\mathcal Z} }
\newcommand{\cH}{{\mathcal H} }
\newcommand{\cK}{{\mathcal K} }
\newcommand{\wh}{\widehat}
\newcommand{\wt}{\widetilde}
\newcommand{\pt}{\partial}
\def\ol#1{{\overline{#1}}}
\newtheorem{theorem}{Theorem}
\newtheorem{definition}{Definition}
\newtheorem{lemma}{Lemma}

\newtheorem{proposition}{Proposition}
\newtheorem{corollary}{Corollary}

\newtheorem*{Maintheorem*}{Main Theorem}
\newtheorem{maintheorem}{Theorem}
\newtheorem*{theorem*}{Theorem}

\def\ke{K{\"a}h\-ler-Ein\-stein }
\def\ks{Ko\-dai\-ra-Spen\-cer }
\def\ka{K{\"a}h\-ler }

\def\wp{Weil-Pe\-ters\-son }

\def\tei{Teich\-mül\-ler }
\def\fs{Fubini-Study }
\def\ii{\sqrt{-1}}
\def\ddb{\sqrt{-1}\partial\overline{\partial}}
\def\C{\mathbb{C}}

\def\cinf{C^\infty}

\def\gab{{g_{\alpha\ol\beta}}}

\def\db{{\ol\partial}}
\def\psh{plurisubharmonic}

\begin{document}

\title[Positivity of $K_{\cX/S}$]{Positivity of relative canonical bundles
for families of canonically polarized manifolds}

\author{Georg Schumacher}
\address{Fachbereich Mathematik und Informatik,
Philipps-Universit\"at Marburg, Lahnberge, Hans-Meerwein-Strasse, D-35032
Marburg,Germany}
\email{schumac@mathematik.uni-marburg.de}
\date{}
\maketitle
\begin{abstract}
Given an effectively parameterized family $f:\cX \to S$ of canonically
polarized manifolds the \ke metrics on the fibers induce a hermitian
metric on the relative canonical bundle $\cK_{\cX/S}$.

We use a global elliptic equation to show that this metric is strictly
positive. For degenerating families we show that the curvature form on the
total space can be extended as a (semi-)positive closed current. By fiber
integration it is shown that the generalized \wp form on the base
possesses an extension as a positive current. In in this situation, the
determinant line bundle associated to the relative canonical bundle on the
total space can be extended. As an application the quasi-projectivity of
the moduli space $\cM_{\text{can}}$ of canonically polarized varieties
follows.
\end{abstract}

\tableofcontents

\section{Introduction}

For any holomorphic family $f: \cX \to S$ of canonically polarized,
complex manifolds, the unique \ke metrics on the fibers define an
intrinsic metric on the relative canonical bundle $\cK_{\cX/S}$. The
construction is functorial in the sense of compatibility with base
changes. By definition, its curvature form has at least as many positive
eigenvalues as the dimension of the fibers indicates.

\begin{Maintheorem*}\label{th:main1}
Let $\cX \to S$ be a holomorphic family of canonically polarized, compact,
complex manifolds, which is nowhere inifinitesimally trivial. Then the
hermitian metric on $\cK_{\cX/S}$ induced by the \ke metrics on the fibers
is strictly positive.
\end{Maintheorem*}

Actually the first variation of the metric tensor in a family of compact
\ke manifolds contains the information about the induced deformation, more
precisely, it contains the harmonic representatives $A_s=
A^\alpha_{s\ol\beta}\pt_\alpha dz^\ol\beta $ of the \ks classes
$\rho(\pt/\pt s)$. The positivity of the hermitian metric will be measured
in terms of a certain global function. Essential is an elliptic equation
on the fibers, which relates this function to the pointwise norm of the
harmonic \ks forms. The ''strict'' positivity of the corresponding
(fiberwise) operator $(\Box + id)^{-1}$, where $\Box$ is the complex
Laplacian, is shown in a direct way.  For families of compact Riemann
surfaces the elliptic equation was previously derived in terms of
automorphic forms by Wolpert \cite{wo}. Later in higher dimensions a
similar equation arose in the work of Siu \cite{siu:canlift} for families
of canonical polarized manifolds.

The positivity of the relative canonical bundle and the methods involved
is closely related to different questions. We will treat the construction
of a positive line bundle on the moduli space of canonically polarized
manifolds proving its quasi-projectivity and the question about its
hyperbolicity.

\begin{maintheorem}\label{th:main2.0}
Let $(\cK_{\cX/\cH},h)$ be the relative canonical bundle on the total
space over the Hilbert scheme, equipped with the hermitian metric induced
by the \ke metrics on the fibers. Then the curvature form extends to the
total space over the compact Hilbert scheme $\ol\cH$ as a positive, closed
current $\omega^{KE}_{\ol \cX}$ with at most analytic singularities.
\end{maintheorem}

The proof depends on Yau's $C^0$-estimates previously used for the
construction of \ke metrics.
\begin{maintheorem}\label{th:main2}
Let $Y$ be a normal space and $Y' \subset Y$ the complement of a closed
analytic, nowhere dense subset. Let $L'$ be a holomorphic line bundle on
$Y'$ together with a (semi-)positive hermitian metric $h'$, which also may
be singular. Assume that the curvature current can be extended to $Y$ as a
positive, closed current $\omega$ and that the line bundle $L'$ possesses
\underline{local} holomorphic extensions to $Y$. Then there exists a
holomorphic line bundle $(L,h)$ with a singular, positive hermitian
metric, whose restriction to $Y'$ is isomorphic to $(L',h')$. The metric
$h$ can be chosen with at most analytic singularities, if $\omega$ has
this property.
\end{maintheorem}
This theorem is first applied to the total space of the universal family
$\cX \to \cH$ over the Hilbert scheme. Next we consider a certain
determinant line bundle of the relative canonical bundle, which is defined
over the open part of the Hilbert scheme. It carries a Quillen metric, and
its curvature form is the generalized \wp form. The latter is equal to the
fiber integral $\int_{\cX/\cH} c_1(\cK_{\cX/\cH},h)^{n+1}$, where $n$ is
the fiber dimension. These quantities descend to the moduli space of
canonically polarized manifolds. Again we see that the curvature form
extends as a closed, positive current.

Concerning the descent to the moduli space, we can say that so far both
the \wp form and the determinant line bundle descend to the moduli space
and possess local extension to a compactification of the moduli space (we
use the fact that the moduli space is an algebraic space). The existence
of a positive \wp current on the compactified moduli space is shown by
means of Siu's decomposition theorem for positive currents, and the
extension of the line bundle follows from the above
Theorem~\ref{th:main2}.

\begin{maintheorem}
The generalized \wp form on the moduli stack of canonically polarized
varieties is strictly positive. It is the Chern form of a determinant line
bundle, equipped with a Quillen metric. A tensor power of the line bundle
extends to a compactification, and the Quillen metric extends as a
(semi-)positive singular hermitian metric with at most analytic
singularities (in the orbifold sense).
\end{maintheorem}

These facts imply a short proof for the {\it quasi-projectivity of the
moduli space of canonically polarized manifolds}: The proof starts from
Mumford's construction of the moduli space as an algebraic space
(possessing a compactification as an analytic space), where the
generalized \wp metric is known to be the Chern form of a certain
determinant line bundle equipped with a Quillen metric. The generalized
Kodaira embedding theorem (cf.\ \cite{s-t}, and the discussion in
\cite{v2}) yields quasi-projectivity, if the determinant line bundle can
be extended to a compactification, and the Quillen metric possesses an
extension as a singular (semi-positive) hermitian metric.

It was pointed out by Eckart Viehweg in \cite[p.~4]{v2} that the extension
of the \wp current is not automatic. This issue is closely related to the
extendability of the determinant line bundle, a problem emphasized by
Kollar \cite{ko1} -- problems, which are addressed in our manuscript.

A further application is the following fact. It is related to
Shafarevich's hyperbolicity conjecture for higher dimensions, which was 
solved by Migliorini \cite{m}, Kovacs \cite{kv1,kv2,kv3}, Bedulev-Viehweg
\cite{b-v}, and Viehweg-Zuo \cite{v-z,v-z2}.

{\bf Application.} {\it Let $\cX \to C$ be a non-isotrivial holomorphic
family of canonically polarized complex surfaces over a curve. Then
$g(C)>1$.}

\section{Positivity of $K_{\cX/S}$}\label{se:posi}
Let $X$ be a canonically polarized manifold of dimension $n$ equipped with
a \ke metric $\omega_X$. In terms of local holomorphic coordinates
$(z^1,\ldots, z^n)$ we write
$$
\omega_X=\ii g_{\alpha\ol\beta}(z)\; dz^\alpha\wedge dz^\ol\beta
$$
so that the \ke equation reads
\begin{equation}\label{eq:ke}
\omega_X=-{\rm Ric}(\omega_X),  \text{ i.e. }  \omega_X= \ddb \log g(z),
\end{equation}
where $g:=\det g_{\alpha\ol\beta}$. We consider $g$ as a hermitian metric
on the anti-canonical bundle $K_X^{-1}$.

For any holomorphic family of compact, canonically polarized manifolds $f:
\cX \to S$ of dimension $n$ with fibers $\cX_s$ for $s\in S$ the \ke forms
$\omega_{\cX_s}$ depend differentiably on the parameter $s$. The resulting
relative \ka form will be denoted by
$$
\omega_{\cX/S} = \ii g_{\alpha,\ol\beta}(z,s)\;dz^\alpha\wedge dz^\ol\beta.
$$
The corresponding hermitian metric on the relative anti-canonical bundle
is given by $g=\det \gab(z,s)$.  We consider the real $(1,1)$-form
$$
\omega_\cX= \ddb \log g(z,s)
$$
on the total space $\cX$. We will discuss the question, whether
$\omega_\cX$ is a \ka form on the total space.

The \ke equation \eqref{eq:ke} implies that
$$
\omega_\cX|\cX_s = \omega_{\cX_s}
$$
for all $s\in S$. In particular $\omega_\cX$, restricted to any fiber, is
positive definite. Our  result is the following statement.

\begin{theorem}\label{th:main}
Let $\cX \to S$ be a holomorphic family of canonically polarized, compact,
complex manifolds. Then the hermitian metric on $\cK_{\cX/S}$ induced by
the \ke metrics on the fibers is semi-positive and strictly positive in
fiber direction. It is strictly positive over points of the base, where
the family is not infinitesimally trivial.
\end{theorem}

Both the statement of the Theorem and the methods are valid for smooth,
proper families of singular (even non-reduced) complex spaces (for the
necessary theory cf.\ \cite{f-s:extremal}).

It is sufficient to prove the theorem for one-dimensional families
assuming $S\subset \C$. (In order to treat singular base spaces, we can
reduce the claim to non-reduced base spaces of embedding dimension one,
where the arguments below are still meaningful and can be applied
literally.)

We denote the \ks map for the family $f:\cX \to S$ at a given point
$s_0\in S$ by
$$
\rho_{s_0} :T_{s_0} \to H^1(X, \cT_X)
$$
where $X=\cX_{s_0}$. The family is called {\it effectively parameterized}
at $s_0$, if $\rho_{s_0}$ is injective. The \ks map is induced as edge
homomorphism by the short exact sequence
$$
0 \to  \cT_{\cX/S} \to \cT_\cX \to f^*\cT_S \to 0.
$$
If $v \in T_{s_0}S$ is a tangent vector, say $v=\frac{\pt}{\pt s}|_{s_0}$
and $\frac{\pt}{\pt s} + b^\alpha \frac{\pt}{\pt z^\alpha}$ is any lift to
$\cX$ along $X$, then
$$
\ol\pt\left(\frac{\pt}{\pt s} + b^\alpha(z) \frac{\pt}{\pt z^\alpha}\right)=
\frac{\partial b^\alpha(z)}{\partial z^\ol\beta}
\frac{\pt}{\pt z^\alpha} dz^\ol\beta
$$
is a $\ol\pt$-closed form on $X$, which represents $\rho_{s_0}(\pt / \pt s)$.
Observe that $b^\alpha$ is not a tensor on $X$ unless the family is
infinitesimally trivial.

We will use the semi-colon notation as well as raising and lowering of indices
for covariant derivatives with respect to the {\it \ke metrics on the fibers}.
The $s$-direction will be indicated by the index $s$. In this sense the
coefficients of $\omega_\cX$ will be denoted by $g_{s\ol s}$, $g_{\alpha\ol
s}$, $\gab$ etc.

Next, we define {\it canonical lifts} of tangent vectors of $S$ as
differentiable vector fields on $\cX$ along the fibers of $f$ in the sense of
Siu \cite{siu:canlift}. By definition these satisfy  the property that the
induced representative of the \ks class is {\it harmonic}  (cf.\ also
\cite{sch:curv}).

Since the form $\omega_\cX$ is positive, when restricted to fibers, {\em
horizontal lifts} of tangent vectors with respect to the pointwise
sesquilinear form $\langle-,-\rangle_{\omega_\cX}$ are well-defined.
\begin{lemma}\label{le:canlift}
The horizontal lift of $\pt/\pt s$  equals
$$
v = \pt_s + a_s^\alpha \pt_\alpha,
$$
where
$$
a_s^\alpha = - g^{\ol\beta \alpha} g_{s \ol \beta}.
$$
\end{lemma}

\begin{proposition}\label{pr:harmrep}
The horizontal lift induces the harmonic representative of
$\rho_{s_0}(\pt/\pt s)$.
\end{proposition}

\begin{proof}
The \ks form of the tangent vector $\pt/\pt_{s_0}$is given by $\ol\pt
v|\cX_s = a^\alpha_{s;\ol\beta}\pt_\alpha dz^\ol\beta$.

The above equation follows immediately: We consider the tensor
$$
A^\alpha_{s\ol\beta}:= a^\alpha_{s;\ol\beta}|{\cX_{s_0}}
$$
on $X$. Then
\begin{gather*}
g^{\ol\beta\gamma} A^\alpha_{s\ol\beta;\gamma}= - g^{\ol\beta\gamma}
g^{\ol\delta \alpha} g_{s\ol\delta;\ol\beta\gamma} = - g^{\ol\beta\gamma}
g^{\ol\delta \alpha} g_{s\ol\beta;\ol\delta\gamma} =
-g^{\ol\beta\gamma}g^{\ol\delta \alpha} \left( g_{s\ol\beta;\gamma\ol\delta} -
g_{s\ol\tau}R^\ol\tau_{\; \ol\beta\ol\delta\gamma} \right)\\
=-g^{\ol\delta\alpha}\left(\left({\pt\log g}/{\pt s} \right)_{;\ol\delta} +
g_{s\ol \tau}R^\ol\tau_{\; \ol\delta}\right) =
0.
\end{gather*}
\end{proof}
Next, we introduce a {\it global} function $\varphi(z,s)$, which is the
pointwise inner product of the canonical lift $v$ of $\pt/\pt s$ at $s\in S$
with itself with respect to $\omega_\cX$. Since $\omega_\cX$ is not known to be
positive definite in all directions, $\varphi\geq 0$ is not known at this
point.
\begin{lemma}\label{le:varphi}
$$
\varphi = \langle \pt_s + a_s^\alpha \pt_\alpha, \pt_s + a_s^\alpha
\pt_\beta  \rangle_{\omega_\cX} = g_{s\ol s} - g_{\alpha\ol s} g_{s\ol\beta}
g{^{\ol\beta\alpha}}
$$
\end{lemma}
\begin{proof}
The proof follows from Lemma~\ref{le:canlift} and
$$
\varphi = g_{s\ol s} + g_{s\ol\beta}a^\ol\beta_{\ol s} + a_s^\alpha g_{\alpha\ol s}
+ a_s^\alpha a_{\ol s}^\ol\beta \gab.
$$
\end{proof}

Denote by $\omega^{n+1}_\cX$ the $(n+1)$-fold exterior product, divided by
$(n+1)!$ and by $dV$ the Euclidean volume element in fiber direction. Then the
global real function $\varphi$ satisfies the following property:
\begin{lemma}
$$
\omega^{n+1}_\cX= \varphi \cdot g \cdot dV\ii ds\wedge \ol{ds}.
$$
\end{lemma}
\begin{proof}
Compute the following $(n+1)\times(n+1)$-determinant of
$$
\left(
\begin{array}{cc}
g_{s\ol s} & g_{s\ol\beta}\\ g_{\alpha\ol s}& \gab
\end{array}
\right),
$$
where $\alpha,\beta=1,\ldots,n$.
\end{proof}

So far we are looking at {\it local} computations, which essentially only
involve derivatives of certain tensors. The only {\it global ingredient}\/ is
the fact that we are given global solutions of the \ke equation.

The key quantity is the differentiable function $\varphi$ on $\cX$. Restricted
to any fiber it ties together the yet to be proven positivity of the hermitian
metric on the relative canonical bundle and the canonical lift of tangent
vectors, which is related to the harmonic \ks forms.

We use the Laplacian operators $\Box_{g,s}$ with non-negative eigenvalues
on the fibers $\cX_s$ so that for a real valued function $\chi$ the
equation $\Box_{g,s}\chi
 = - g^{\ol\beta\alpha}\chi_{;\alpha\ol\beta}$ holds.
\begin{proposition}\label{pr:elleq}
The following elliptic equation holds fiberwise:
\begin{equation}\label{eq:phiA}
(\Box_{g,s} + {\rm id})\varphi(z,s) = \|A_s(z,s)\|^2,
\end{equation}
where
$$
A_s=A^\alpha_{s\ol\beta} \frac{\pt}{\pt z^\alpha}dz^\ol\beta
$$
is the harmonic representative of the \ks class $\rho_s(\frac{\pt}{\pt
s})$ as above.
\end{proposition}
\begin{proof}
The essence to prove an elliptic equation for the tensors that involve
derivatives with respect to the parameter space is to eliminate second
order such derivatives. This is achieved by the left hand side of
\eqref{eq:phiA}. First,
\begin{eqnarray*}
g^{\ol\delta\gamma}g_{s\ol s;\gamma\ol\delta}
&=&g^{\ol\delta\gamma}\partial_s\partial_\ol s g_{\gamma\ol\delta}\\
&=&\partial_s(g^{\ol\delta\gamma}\partial_\ol s g_{\gamma\ol\delta})
 -a_s^{\gamma;\ol\delta}\partial_\ol s g_{\gamma\ol\delta}\\
&=&\partial_s\partial_\ol s \log g
 +a_s^{\gamma;\ol\delta} a_{\ol s\gamma;\ol\delta}\\
&=& g_{s \ol s}
 +a_s^\sigma{}_{;\gamma} a_{\ol s\sigma;\ol\delta} g^{\ol\delta\gamma}.
\end{eqnarray*}
Next
\begin{eqnarray*}
(a_s^\sigma a_{\ol s\ol\sigma})_{;\gamma\ol\delta}g^{\ol\delta\gamma}
&=\left(a_s^\sigma{}_{;\gamma\ol\delta} a_{\ol s\sigma}
        +A_{s\ol\delta}^\sigma A_{\ol s\sigma\gamma}
        +a_{s;\gamma}^\sigma a_{\ol s\sigma;\ol\delta}
        +a_s^\sigma A_{\ol s\sigma\gamma;\ol\delta}
\right)g^{\ol\delta\gamma} .
\end{eqnarray*}
The last term vanishes because of the harmonicity of $A_s$, and
\begin{eqnarray*}
a_{s;\gamma\ol\delta}^\sigma g^{\ol\delta\gamma}
&=&A_{s\ol\delta;\gamma}^\sigma g^{\ol\delta\gamma}
  +a_s^\lambda R^\sigma{}_{\lambda\gamma\ol\delta}g^{\ol\delta\gamma}\\
&=&0-a_s^\lambda R^\sigma{}_\lambda\\
 &=& a_s^\sigma .
\end{eqnarray*}
\end{proof}

\begin{proof}[Proof of Theorem~\ref{th:main}] We first show the semi-positivity
of the metric.

As $\omega_\cX$ is positive definite in fiber direction, we need to show
only that $\varphi \geq 0$ (or $\varphi >0$ resp.) For any fixed $s\in S$,
let
$$
\varphi(z,s) \geq \varphi(z_0,s).
$$
Then
$$
\varphi(z_0,s)= \|A_s(z_0,s)\|^2 - \Box_{g,s}\varphi(z_0,s) \geq  \|A_s(z_0,s)\|^2 \geq
0.
$$
For any fixed $s\in S$ the function $\varphi|\cX_s$ is not identically
zero, otherwise by \eqref{eq:phiA} the family had to be infinitesimally
trivial at that point.

According to a the theorem of Kazdan and De Turck \cite{kdt}, \ke metrics are
real analytic (and by the implicit function theorem depend in a real analytic
way upon holomorphic parameters). This applies to the function $\varphi$.

The above argument shows that the zero set of $\varphi$ is contained in the set
of points, where all components of $A_s$ vanish.

We mention that the integral mean of $\omega^{n+1}_\cX$ taken over the fibers
is equal to the generalized \wp form on $S$ (cf.\ \cite[Theorem
7.9]{f-s:extremal}).

The strict positivity of $\varphi$ follows from the proposition below.
\end{proof}

We consider the equation \eqref{eq:phiA} locally. Let $0\in U \subset \C^n$ be
an open subset containing the origin, and $\omega_U = \frac{\ii}{2}
g_{\alpha\ol\beta}(z)dz^\alpha\wedge dz^\ol\beta$ a real analytic \ka form on
$U$.
\begin{proposition}
Let $\varphi$ and $f$ be real analytic, non-negative, real functions on $U$.
Suppose
\begin{equation}\label{eq:main}
\Box_{\omega_U}\varphi + \varphi = f
\end{equation}
holds. If $\varphi(0)=0$, then both $\varphi$ and $f$ vanish identically in a
neighborhood of $0$.
\end{proposition}
\begin{proof}\footnote{The claim is also a consequence of
\cite[Theorem 6, Chap.~2, Sect.~3]{pw}.} It follows from the assumption
that $\varphi$ has a local minimum at the origin, and \eqref{eq:main}
implies that $\Box_{\omega_U} \varphi(0)=0$ and $f(0)=0$.

We set $\Box=\Box_{\omega_U}$ and chose normal coordinates $z^\alpha$ of the
second kind for $\omega_U$ at $0$. Let $\Box_0= - \sum_{\alpha=1}^n
\frac{\pt^2}{\pt z^\alpha\pt z^\ol\alpha}$ be the standard Laplacian so that
$$
\Box=\Box_0 + h^{\ol\beta\alpha}(z)\frac{\pt^2}{\pt z^\alpha\pt z^\ol\beta}
$$
where the power series expansions of all $h^{\ol\beta\alpha}$ have no
terms of order zero or one. Also $\Box_0\varphi(0)=0$. In the following
arguments it may be necessary to replace $U$ by a smaller neighborhood of
zero. Then we can say that both $\Box \varphi$ and $\Box_0\varphi$ are
non-positive.

We suppose that $\varphi$ is not identically zero and let
$$
\varphi= \sum_{\ell\geq\ell_0} \varphi_\ell
$$
be the homogeneous expansion of $\varphi$ into polynomials of degree $\ell$
with $\varphi_{\ell_0}\neq 0$. It follows from the assumption that
$\varphi_{\ell_0} \geq 0$.

The homogeneous components of the Laplacians of least possible order are
the components of degree $\ell_0-2$
$$
(\Box\varphi)_{\ell_0 -2}=(\Box_0\varphi)_{\ell_0 -2}
$$
because of the choice of the coordinates. Suppose that $(\Box_0
\varphi)_{\ell_0 -2}= \Box_0(\varphi_{\ell_0})$ vanishes identically. Then
the mean value property implies that $\varphi_{\ell_0}\equiv 0$, which
contradicts the choice of $\ell_0$.

Now the integral over the sphere $S(r)$ of radius $r$ with respect to the
standard (flat) inner product and surface element $dA$ is taken. Let
$$
\wt\varphi(r)= \int_{S(r)} \varphi dA.
$$
Then
$$
0 \geq\int_{S(r)} \Box\varphi dA=\int_{S(r)} \Box_0\varphi dA + R(r)
$$
where the remaining term $R(r)$ is of order at least $\ell_0+2n-1$ in $r$,
whereas the integrals are of order $\ell_0+ 2n - 3$, unless they vanish
identically. In the latter case, again the Laplacians are identically zero,
implying $\varphi\equiv 0$. We consider the integrated equation
\eqref{eq:main}. The order of $\wt\varphi(r)$ is $\ell_0 + 2n -1$ so that the
lowest order term on the left hand side of
$$
\int_{S(r)} \Box\varphi dA + \wt \varphi(r) = \int_{S(r)} f dA
$$
is $c \cdot r^{\ell_0 + 2n -3}$, with $c<0$ contradicting the property of $f$.
\end{proof}

\section{Fiber integrals and Quillen metrics}\label{se:bgs}
In this section we refer to the methods how to produce a positive line
bundle on the base of a holomorphic family (cf.\ \cite{f-s:extremal}). Let
$f:\cX \to S$ be a proper, smooth holomorphic map and $\omega_{\cX/S}$ a
closed real $(1,1)$-form on $\cX$, whose restrictions to the fibers are
\ka forms. Let $(\cE, h)$ be a hermitian vector bundle on $\cX$. We denote
the determinant line bundle of $\cE$ in the derived category by
$$
\lambda(\cE)=\det f_!(\cE).
$$
The main result of Bismut, Gillet and Soulé from \cite{bgs} states the
existence of a Quillen metric $h^Q$ on the determinant line bundle such
that the following equality holds for its Chern form on the base $S$:
\begin{equation}\label{eq:bgs}
c_1(\lambda(\cE),h^Q)= \left[\int_{\cX/S}\textit{td}(\cX/S,\omega_{\cX/S})\textit{ch}(\cE,h)\right]_2
\end{equation}
holds. Here $\textit{ch}$ and $\textit{td}$ stand for the Chern and Todd
character forms.

We will apply the formula in two different situations.

We use the formula for a virtual bundle of degree zero and set $\cE=(\cL-
\cL^{-1})^{n+1}$, where $(\cL,h)$ is a hermitian line bundle. The term of
lowest degree in $\text{ch}(\cE)$ is equal to
\begin{equation}\label{eq:fib}
2^{n+1} c_1(\cL)^{n+1}
\end{equation}
so that the only contribution of the Todd character form in \eqref{eq:bgs}
is the constant $1$ resulting in the following equality.
\begin{equation}\label{eq:fib0}
c_1(\lambda(\cE),h^Q) = 2^{n+1} \int_{\cX/S}c_1(\cL,h)^{n+1}.
\end{equation}

\section{An extension theorem for hermitian line
bundles\\and positive currents}\label{se:extlinebundles}

We will need the notion of a singular hermitian metric over a (reduced)
complex space. First, we note that by definition an upper semi-continuous
function $u: \cZ \to [-\infty,\infty)$ on a complex space $\cZ$ is
plurisubharmonic, if its pull-back to any (locally given) analytic curve
is subharmonic (or identically equal to $-\infty$). Let $\cL$ be a
holomorphic line bundle on $\cZ$, then a semi-positive, singular hermitian
metric $h$ on $\cL$ is defined by the property that the locally defined
function $-\log h$ is plurisubharmonic, when pulled back to the
normalization of the space. Only in this sense we consider positive
$(1,1)$-currents on reduced complex spaces. As usual by definition, a {\em
positive current} takes non-negative values on semi-positive differential
forms.

The following is an existence theorem for a holomorphic line bundle, while
there seems to be little control over the actual extension.

For any positive closed $(1,1)$-current $T$ on a complex manifold $Y$ the
Lelong number at a point $x$ is denoted by $\nu(T,x)$, and for any $c>0$
we have the associated sets $E_c(T)=\{ x; \nu(T,x)\geq c\}$. According to
\cite[Main Theorem]{siu:curr} these are closed analytic sets.

We will use in an essential way Siu's decomposition formula for  positive,
closed currents. We state it for $(1,1)$-currents.
\begin{theorem*}[{\cite{siu:curr}}]
Let $\omega$ be a closed positive $(1,1)$-current. Then $\omega$ can be
written as a convergent series of closed positive currents
\begin{equation}\label{eq:decomp}
\omega = \sum_{\mu=0}^\infty \mu_k [Z_k] + R,
\end{equation}
where $[Z_k]$ is a current of integration over an irreducible analytic set
of codimension one, and R is a residual current with the property that
$\dim E_c(R) < \dim Y-1$ for every $c > 0$. This decomposition is locally
and globally unique: the sets $Z_k$ are precisely components of
codimension one occurring in the sublevel sets $E_c(\omega)$, and $\mu_k =
\min_{x\in Z_k} \nu(\omega; x)$ is the generic Lelong number of $\omega$
along $Z_k$.
\end{theorem*}

Our extension theorem for holomorphic line bundles is the following.

\begin{theorem}\label{th:extlinebdl}
Let $Y$ be a normal complex space and $Y' = Y \setminus A$ the complement
of a nowhere dense, closed, analytic subset. Let $L'$ be a holomorphic
line bundle together with a (semi-)positive hermitian metric $h'$, which
also may be singular. Assume
\begin{itemize}
  \item[(i)] the curvature current $\omega'$ of $(L',h')$ possesses an
      extension $\omega$ to $Y$ as a closed, positive current
      \item[(ii)] the line bundle $L'$ possesses \underline{local}
          holomorphic extensions to $Y$.
\end{itemize}
Then there exists a holomorphic line bundle $(L,h)$ with a singular,
positive hermitian metric, whose restriction to $Y'$ is isomorphic to
$(L',h')$.

If $\omega$ has at most analytic singularities, then $h$ can be chosen
with this property.
\end{theorem}

{\em Remark.} The form $\omega$ need not be the curvature form of $h$.

\medskip

We append the notion of analytic singularities.

\begin{definition} Let $\chi$ be a plurisubharmonic function on an open subset of
$\C^n$. We say that $\chi$ has {\em at most analytic singularities}, if
locally
$$
\chi \geq \gamma\log \sum^{k}_{\nu=1} |f_\nu|^2 + const.
$$
holds, for holomorphic functions $f_\nu$ and some $\gamma>0$. In this
situation, we say that a (locally defined) positive, singular hermitian
metric of the form
$$
h=e^{-\chi}
$$
as at most analytic singularities. This property will be also assigned to
a locally $\pt\ol\pt$-exact, positive current of the form
$$
\omega=\ii \pt\ol\pt \chi.
$$
\end{definition}

\begin{proof}[Proof of Theorem~\ref{th:extlinebdl}]
We first assume that $A$ is a smooth hypersurface of a complex manifold
$Y$.

Our goal is to treat the non-integer part
$$
\sum_{\mu=0}^\infty (\mu_k - \lfloor \mu_k\rfloor)[Z_k]
$$
in \eqref{eq:decomp}.

Let $\{U_j \}$ be an open covering of $Y$ such that the set $A\cap U_i$
consists of the zeroes of a holomorphic function $z_i$ on $U_i$. Since
$L'$ possesses local extensions, we can chose $\{U_i\}$ such that the line
bundle $L'$ is given by a cocycle $g_{ij}' \in \cO_Y^*(U_{ij}')$, where
$U_{ij}=U_i\cap U_j$ and $U_{ij}'=U_{ij}\cap Y'$.  If necessary, we will
replace $\{U_i\}$ by a finer covering.

We will first show the existence of \psh\ functions $\psi_i$ on $U_i$ and
holomorphic functions $\varphi'_i$ on $U'_i$ such that
$$
h'_i\cdot|\varphi'_i|^2 = e^{-\psi_i}|U'_i
$$
where the $\psi_i-\psi_j$ are harmonic functions on $U_{ij}$:

Let
$$
\omega|U_i= \ddb (\psi^0_i)
$$
for some \psh\ functions $\psi^0_i$ on $U_i$. Now
$$
-\ddb \log(e^{\psi^0_i} h'_i)
$$
is harmonic on $U'_i$. For a suitable number $\beta_i\in \R$ and some
holomorphic function $f'_i$ on $U'_i$ we have
$$
\log(e^{\psi^0_i} h'_i) + \beta_i \log|z_i| = f'_i+ \ol{f'_i}.
$$
We write
$$
\beta_i=\gamma_i + 2k_i
$$
for $0\leq \gamma_i <2$ and some integer $k_i$. We set
\begin{equation}\label{eq:psi}
\psi_i = \psi^0_i + \gamma_i \log |z_i|.
\end{equation}
These functions are clearly \psh, and $\gamma_i \log |z_i|$ contributes as
an analytic singularity to $\psi^0_i$. Set
$$
\varphi'_i= z^{-k_i}_i e^{-f'_i} \in \cO^*(U'_i).
$$
We use the functions $\varphi'_i$ to change the bundle coordinates of $L'$
with respect to $U'_i$. In these bundle coordinates the hermitian metric
$h'$ on $L'$ is given by
$$
\wt h'_i = h'_i \cdot |\varphi'_i|^{-2}
$$
and the transformed transition functions are
$$
\wt g'_{ij}= {\varphi'_i}\cdot g_{ij}\cdot ({\varphi'_j})^{-1}.
$$
Now
\begin{equation}\label{eq:wth}
\wt h'_i = |z_i|^{-\gamma_i} e^{-\psi^0_i}|U'_i= e^{-\psi_i}|U'_i ,
\end{equation}
and
\begin{equation}\label{eq:wtg}
|\wt g'_{ij}|^2 = \wt h'_j (\wt h'_i)^{-1} =
|z_j|^{\gamma_i-\gamma_j} \left|\frac{z_i}{z_j}\right|^{\gamma_i} \cdot
e^{\psi^0_i-\psi^0_j}.
\end{equation}
Since the function $z_i/z_j$ is holomorphic and nowhere vanishing  on
$U_{ij}$, and since the function $\psi^0_i-\psi^0_j$ is harmonic on
$U_{ij}$, the function
$$
    \left|\frac{z_i}{z_j}\right|^{\gamma_i} \cdot
e^{\psi^0_i-\psi^0_j}
$$
is of class $\cinf$ on $U_{ij}$ with no zeroes. Now
$-2<\gamma_i-\gamma_j<2$, and $\wt g'_{ij}$ is holomorphic on $U'_{ij}$.
So we have
\begin{equation}\label{eq:beta}
\gamma_i=\gamma_j
\end{equation}
(whenever $A\cap U_{ij} \neq \emptyset)$. Accordingly \eqref{eq:wtg} reads
\begin{equation}\label{eq:wtg1}
|\wt g'_{ij}|^2 = \wt h'_j (\wt h'_i)^{-1} =
\left|\frac{z_i}{z_j}\right|^{\gamma_i} \cdot
e^{\psi^0_i-\psi^0_j},
\end{equation}
and the transition functions $\wt g'_{ij}$ can be extended holomorphically
to all of $U_{ij}$. So a line bundle $L$ exists.

The functions $\psi_i$ are \psh, and the equality
$$
\wh \omega = \ddb \psi_i
$$
is a well-defined positive current on $Y$ because of \eqref{eq:beta}. Its
restriction to $Y'$ equals $\omega|Y'$.

We define
$$
\wt h_i= e^{-\psi_i}
$$
on $U_i$. Because of \eqref{eq:wtg} it defines a positive, singular,
hermitian metric on $L$. This shows the theorem in the smooth case.

Finally we prove the general case, where $Y$ is normal and $A$ any nowhere
dense analytic subset. Denote by $A^0$ the union of all components of $A$,
whose dimension is less or equal to $\dim Y-2$. Let
$$
Y''= Y\backslash \left({\rm Sing }(Y) \cup  {\rm Sing }(A) \cup A^0\right).
$$
We know that $L'$ together with the positive metric can be extended to a
line bundle $L''$ on $Y''$. Again we use the fact that the given line
bundle possesses local extensions to $Y$ so that $L''$ is defined by
transition functions on $U_{ij}\cap Y''$, where again the $U_i$ form an
open covering of $Y$. Now the complement of $Y''$ in $Y$ is of codimension
at least two so that the transition functions extend. The same argument
holds for the singular hermitian metrics $\wt h''_i$ of $L''$ on $U_i \cap
Y''$: The \psh\ functions $-\log \wt h''_i$ extend as such to $U_i$ by
\cite{g-r}.
\end{proof}

We remark that the theorem also holds for orbifold structures (where a
certain finite tensor power of the given line bundles may have to be taken
in order to ensure that the bundles descend as holomorphic line bundles).
This is globally possible as long as the automorphism groups involved are
of bounded order, which is the case for moduli spaces of canonically
polarized varieties.

If $Y$ is just a reduced complex space, we still have the following
statement.

\begin{proposition}\label{pr:nonnormal}
Let $Y$ be a reduced complex space, and $A\subset Y$ a closed analytic
subset. Let $\cL$ be an invertible sheaf on $Y \backslash A$, which
possesses a holomorphic extension to the normalization of\/ $Y$ as an
invertible sheaf. Then there exists a reduced complex space $Z$ together
with a finite map $Z \to Y$, which is an isomorphism over $Y\backslash A$
such that $\cL$ possesses an extension as an invertible sheaf to $Z$.
\end{proposition}
\begin{proof}
Denote by $\nu:\wh Y \to Y$ the normalization of $Y$. The presheaf
$$
U\mapsto \{\sigma \in  (\nu_* \cO_{\wh Y})(U); \sigma|U\backslash A\in \cO_Y(U\backslash A)\}
$$
defines a coherent $\cO_Y$-module, the so called  {\em gap sheaf}
$$
\cO_Y[A]_{\nu_*\cO_{\wh Y}}
$$
on $Y$ (cf.\ \cite[Proposition 2]{siu:gap}). It carries the structure of
an $\cO_Y$-algebra. According to Houzel \cite[Prop.\ 5 and Prop.\ 2]{hou}
it follows that it is an $\cO_Y$-algebra of finite presentation, and hence
its analytic spectrum provides a complex space $Z$ over $Y$ (cf.\ also
Forster \cite[Satz 1]{fo}).
\end{proof}
Finally, we have to deal with the question of a {\em global} extension of
positive currents. Siu's Thullen-Remmert-Stein type theorem \cite[Theorem
1]{siu:curr} gives an answer for currents on open sets in a complex number
space. An extension for $(1,1)$-currents exists and is uniquely determined
by one local extension into each irreducible hypersurface component of the
critical set. In a global situation we need an extra argument, namely the
fact that (under some assumption), one can single out an extension, which
is unique and not depending upon the choice of some local extension.

We will show the following extension argument for positive currents given
on a normal complex space $Y$. Let $A\subset Y$ a closed analytic subset
and $Y'=Y\backslash A$.
\begin{proposition}\label{pr:extcurr}
Let $\omega'$ be a closed, positive current on $Y'$, whose Lelong numbers
vanish everywhere. Assume that for any point of $A$ there exists an open
neighborhood $U\subset Y$ such that $\omega'$ possesses an extension to
$Y'\cup U$ as a closed positive current. Then $\omega$ can be extended to
all of $Y$.

If $\omega'$ has at most analytic singularities, and if this is also true
for the local extensions, then also the global extension has at most
analytic singularities.
\end{proposition}
\begin{proof}
We first remove the critical sets from $Y$: Let
$$
Y''=Y\backslash \left( {\rm Sing}(Y)\cup {\rm Sing}(A)\cup A^0\right),
$$
where $A^0$ is the union of all components of $A$ whose dimension is less
or equal to $\dim Y-2$.

Let $p\in Y''\cap A$. so $Y$ is non-singular at $p$ and $p\in A$ is in the
smooth hyperplane  part of $A$.

We assume that $\omega'|Y''\cap Y'$ can be extended as a positive closed
current into some open neighborhood $U$ of $p$. Denote by $\omega_U
\subset Y''$ the extension to to $U\cup Y''$.

We apply \eqref{eq:decomp}
$$
\omega_U = \sum_{\mu=0}^\infty \mu_k [Z_k] + R\, .
$$
Since the Lelong numbers of $\omega'$ vanish everywhere, the sets $Z_k$
must be contained in $A$ so that the residual current $R$ is an extension
of $\omega'$.

Let $\breve \omega_U$ be an other extension, which can also be reduced to
its residual current $\breve R$. On some neighborhood $W$ of $p$ we write
$$
R|W = \ddb \chi\quad  \text{and}\quad \breve R|W = \ddb \breve\chi.
$$
Now
$$
\ddb (\chi-\breve \chi)|W\backslash A =0.
$$
We pick some $0<c<1$ and remove the set $E_c(\chi)\cup E_c(\breve\chi)$,
which is of codimension at least two and replace $W$ by and open subset
which does not intersect $E_c(\chi)\cup E_c(\breve\chi)$.

Let $A\cap W=V(z)$ for some coordinate function $z$ on $W$. As
$\chi-\breve\chi$ is harmonic on $W\backslash A$,
$$
\left(\chi-\breve \chi -\beta \log |z|\right)|W\backslash A = f' + \ol f'
$$
for some holomorphic function $f'\in\cO(W\backslash A)$ and $\beta \in
\R$.

Now by the theorem of Skoda and Bombieri \cite{bom,sk} both
$e^{2(\chi-\breve\chi)}$ and $e^{2(\breve\chi-\chi)}$ are in
$L^1_{loc}(W)$. We have
$$
e^{\chi-\breve\chi} = |z|^\beta |e^{f'}|^2.
$$
On can see that $f'$ can only have  poles or zeroes on $A\cap W$. Next, we
infer from the integrability conditions that
$$
e^{\chi-\breve\chi} = |z|^\gamma |\varphi|^2,
$$
where $\varphi$ is holomorphic on all of $W$ with no zeroes and
$-1<\gamma<1$. If present, the contribution $|z|^\gamma$ would constitute
a current of integration along a (local) hypersurface. Since the
decomposition formula for positive currents is compatible with restriction
to open subsets, a current of integration is not present. Hence
$\omega_U|W=\breve\omega_U|W$. The last step only contains the extension
of \psh\ functions on normal complex spaces into subsets of codimension at
least two.
\end{proof}

\section{\wp metric on Hilbert schemes\\ and Douady
spaces}\label{se:wphilb}

In \cite{b-s} and \cite{a-s} a \wp metric for Douady spaces and Hilbert
schemes was studied. Let $(Z,\omega_Z)$ be a \ka manifold, and
\begin{equation}
\xymatrix{{\cX}\ar@{^(->}[r]^-{i}\ar[dr]_f &Z\times
S\ar[d]^{{\rm{pr_2}}}
\\ & S}
\end{equation}
a flat holomorphic family of complex submanifolds of $Z$, parameterized by
a complex space $S$ of dimension $n$. Let $s_0$ be a distinguished point
of $S$ with fiber $X=\cX_{s_0}$. The associated \ks map is
$$
\rho_{s_0}: T_{s_0}S \to H^0(X, \cN_{X|Z}).
$$
Denote by $\omega_X$ the induced \ka form on $X$. The \ka form $\omega_Z$
induces a pointwise hermitian inner product
$\langle\cdot,\cdot\rangle_{\omega_Z}$ on the normal bundle $\cN_{X|Z}$,
and by integration a natural inner product on the space of its global
holomorphic sections.
\begin{definition}\label{pwdef}
The \wp inner product for $v,w \in T_{s_0}S$ equals
$$
\langle v,w\rangle_{PW} = \int_X
\langle\rho(v),\rho(w)\rangle_{\omega_Z} \omega_X^n.
$$
\end{definition}
We recall that the above hermitian inner product is positive definite in
directions, where the given family is effective, and the corresponding
hermitian form on $S$ is induced by the fiber integral
\begin{equation}\label{eq:fibinthilb}
 \omega_S^{PW} = \int\limits_{\cX/S}
(\wt \omega^{n+1}|\cX),
\end{equation}
where $n = \dim X$ and $\wt\omega$ is the pull-back of $\omega_Z$ to
$Z\times S$. The construction is intrinsic, in particular the fiber
integral commutes with base change morphisms. Again the fiber integral
formula implies that the \wp form $\omega_S^{PW}$ is Kähler.

For Hilbert schemes $S=\cH$ and $Z=\mathbb P_N$ equipped with the \fs
metric $\omega_Z= \omega^{FS}=c_1(\cO_{\mathbb P_N}(1), h^{FS})$, the line
bundle $\cL=\cO_{\mathbb P_N}(1)$ gives rise to the determinant line
bundle in the derived category $\lambda=\det f_!((\cL - \cL^{-1})^{n+1})$.
We invoke the results from Section~\ref{se:bgs} and get
$$
\omega^{WP}_\cH = \frac{1}{2^{n+1}} c_1(\lambda,h^Q).
$$
Next, we consider the component $\ol\cH$ of the Hilbert scheme
corresponding to flat (possibly singular) morphisms.
$$
\xymatrix{\ol\cX \ar[r]^{\ol i} \ar[dr]_{\ol f} & \mathbb P_N \times
\ol \cH \ar[d]^{pr_2}\\
& \ol\cH
}
$$
(with flat proper morphism $\ol f$). With $\ol \cL = \ol \imath^*
\cO_{\mathbb P_N}(1)$ we have an extension $\ol\lambda=\det \ol
f_!((\ol\cL-{\ol\cL}^{-1})^{n+1})$ of $\lambda$ as a holomorphic line
bundle to $\ol \cH$. We denote the extension of the \wp metric given by a
fiber integral analogous to \eqref{eq:fibinthilb} by
\begin{equation}\label{eq:extfibint}
 \omega_{\ol\cH} = \int\limits_{\ol\cX/\ol\cH}
(\wh\omega^{n+1}|\ol\cX),
\end{equation}
where $\wh\omega$ is the pull-back of the \fs form to $\ol\cX$.

Generally speaking, the fiber integral is also meaningful for flat
embedded families, where only the existence of one non-singular fiber is
required. It was shown in \cite[Lemme 3.4]{va1} that $\omega_\ol\cH$ is a
$d$-closed real $(1,1)$-current (positive in the sense of currents), which
possesses a {\em continuous $\pt\ol\pt$-potential}. So the curvature form
of the Quillen metric on $\lambda$ extends in this sense. From now on we
restrict ourselves to the normalization of the Hilbert scheme and of its
compactification resp.

The proof of Theorem~\ref{th:extlinebdl} implies the following fact:

\begin{proposition}\label{pr:Hilbquil}
There exists a holomorphic line bundle $\wt\lambda$ on $\ol\cH$ with
$\wt\lambda|\cH\simeq \lambda$.  The Quillen metric extends as a positive,
singular hermitian metric $\wt h^Q$ on $\wt\lambda$ with analytic
singularities and Lelong numbers less than one such that
$$
\omega_{\ol\cH}= \frac{1}{2^{n+1}} c_1(\wt\lambda,\wt h^Q).
$$
\end{proposition}
Note that the line bundle $\wt\lambda$ is only of auxiliary nature. We
will use mainly the current that is given by the fiber integral.

Based upon the results of Bismut from \cite{bi1} also for certain nodal
singularities the Quillen metric of a determinant line bundle can be
defined allowing for such singularities.

Next we will use the \fs metric for embedded families a background metric
in the relative \ke case.

\section{Degenerating families of canonically polarized
varieties}\label{se:degfam} In this section we want to show that in a
degenerating family the curvature of the relative canonical bundle can be
extended as positive closed currents. By definition this is an extension
to the normalization. Accordingly, Hilbert schemes (and compactifications)
are taken in the category of normal spaces.

Given a canonically polarized manifold $X$ of dimension $n$ together with
an $m$-canonical embedding $\Phi= \Phi_{mK_X}: X \hookrightarrow \mathbb
P_N$, the \fs metric $h^{FS}$ on the hyperplane section bundle
$\cO_{\mathbb P_N}(1)$ defines a volume form
\begin{equation}\label{eq:Omega0}
\Omega^0_X= (\sum_{i=0}^N |\Phi_i(z)|^2)^{1/m}
\end{equation}
on the manifold $X$, such that
$$
\omega^0_X:=-\text{Ric}(\Omega^0_X) = \frac{1}{m} \; \omega^{FS}|X,
$$
where $\omega^{FS}$ denotes the Fubini-Study form on $\mathbb P_N$.
According to Yau's theorem, $\omega^0_X$ can be deformed into a \ke metric
$\omega_X=\omega^0_X +  \ddb u$. It solves the equation \eqref{eq:ke},
which is equivalent to
\begin{equation}\label{eq:ke1}
\omega^n_X=(\omega^0_X+\ddb u)^n= e^u\Omega^0_X.
\end{equation}

We will need the $C^0$-estimates for the (uniquely determined)
$\cinf$-function $u$.

The deviation of $\omega^0_X$ from being \ke is given by the function
\begin{equation}\label{eq:defF}
F=\log\frac{\Omega^0_X}{(\omega^0_X)^n}
\end{equation}
so that \eqref{eq:ke1} is equivalent to
\begin{equation}\label{eq:MA}
\omega^n_X=(\omega^0_X+\ddb u)^n= e^{u+F}(\omega^0_X)^n.
\end{equation}
We will use the $C^0$-estimate for $u$ from \cite[Proposition
4.1]{cheng-yau} (cf.\ \cite{aub,kob}).

\bigskip

\noindent {\bf $C^0$-estimate.} \textit{Let $\Box^0$ denote the complex
Laplacian on functions with respect to $\omega^0_X$ (with non-negative
eigenvalues). Then
\begin{equation}\label{eq:uplusF}
 u+F \leq -\Box^0(u).
\end{equation}
In particular the function $u$ is bounded from above by $\sup(-F)$.}

\bigskip


Now we come to the relative situation. Let
\begin{equation}\label{eq:embfam}
\xymatrix{{\ol\cX}\ar[r]^-{\Phi}\ar[dr]_f & \mathbb P_N\times
\ol\cH\ar[d]^{{\rm{pr_2}}}
\\ & \ol\cH}
\end{equation}
Next, we desingularize $\ol \cH$ in a way that the union of the singular
locus and the singular set corresponds to a normal crossings divisor $D
\subset \wt \cH$. We just take the pull-back $\wt \cX =
\ol\cX\times_\ol\cH \wt\cH \to \wt \cH$ of the embedded flat singular
family, allowing non-reduced fibers. Let $\sigma$ be a canonical section
of $[D]$. We use $\sigma$ as a local holomorphic coordinate on $\wt \cH$
and denote by $|\sigma|$ its absolute value.

Let $\wt \Phi : \wt\cX \to \mathbb P_N \times \wt \cH$ be the induced
morphism. Let $\cH'= \wt \cH \backslash D$, and denote by $f':\wt\cX' \to
\cH'$ the restricted map. We set $\cL=\Phi^{*}\cO_{\mathbb P_N\times
S}(1)$ and denote by $\wt\cL$ its pull-back to $\wt\cX$. Furthermore, we
assume that $\wt\cL|\wt \cX'= \cO_{\wt\cX'}(m K_{\wt\cX'/\cH'})$.

Next, we denote the relative initial volume form on $\wt\cX'$ by
$\Omega^0$, which again is defined in terms of $\Phi$ and the \fs metric
on $\cO_{\mathbb P_N}(1)$ like in \eqref{eq:Omega0}. In a similar way
$\omega^0_{\wt\cX'}$ is defined. We denote by $u$ and $F$ the functions
defined above, depending on the base parameter. Since both the initial
volume form $\Omega^0$ and the volume form induced by the relative
Fubini-Study form are induced by polynomials, it follows immediately that
for some positive exponents $k$ and $\ell$ we have
$$
| \sigma(s) |^{2k}
\sup\left(\left.\frac{(\omega^0_{\wt\cX'})^n}{\Omega^0}\right|{\wt\cX'_s}\right)
\leq \textit{const.}
$$
for $s\in \cH'$, i.e.
$$
| \sigma(s) |^{2k} \sup\{ e^{-F(z)}; f(z)=s\} \leq \textit{const.}
$$
and
$$
| \sigma(s) |^{2\ell} \Omega^0 \leq \textit{const.}
$$
holds. Observe that we interpret $(\Omega^0)^m$ here as a hermitian metric
on the line bundle $\cL^{-1}$, which exists on $\ol \cX$. So the above
estimate is meaningful locally around points of $\ol\cX\backslash\cX$.
Note that $e^u\Omega^0$ is a global object though. In this sense the
$C^0$-estimate \eqref{eq:uplusF} implies that
$$
| \sigma(s) |^{2(k+\ell)} e^u\Omega^0 \leq \textit{const.}
$$
The constant can be chosen uniformly,  i.e.\ on a set of the form $\wt
f^{-1}(U\backslash D)$. Over the complement of the divisor $D$ its
logarithm is bounded from above. So far we argued fiberwise. Now by
\eqref{eq:ke1} we have
$$
\omega^n_{\cX/S} = e^u\Omega^0,
$$
(where the right hand side depends on the parameter).

At this point we apply Theorem~\ref{th:main}: We see that $\log
e^u\Omega^0$ is plurisubharmonic (over the regular locus $\wt
f^{-1}(U\backslash D)$). So the logarithm of $| \sigma(s) |^{2(k+\ell)}
e^u\Omega^0$ possesses a plurisubharmonic extension. Now the {\it current}
$$
\ddb\log(| \sigma(s) |^{2(k+\ell)} e^u\Omega^0 )|_{\wt f^{-1}(U\backslash D)}=
\ddb\log( e^u\Omega^0 )|_{\wt f^{-1}(U\backslash D)}
$$
can be extended to $\wt \cX$ as a positive, closed current. (A necessary
term to be added will essentially be some multiple of the current of
integration over the boundary divisor, which we cannot control at this
point.) The push-forward to $\cX$ of the extended current is a positive,
closed extension of the original curvature current. Note that on the
interior the process of pulling back and taking the push forward does not
change the current (being defined on the normalization).

\medskip

We used decisively the fact that $\Omega^0$ has at most algebraic/analytic
singularities and the upper $C^0$-estimates for the Monge-Ampère equation
\eqref{eq:MA}. These imply that the constructed current has at most
analytic singularites.

In this section we proved that $\omega_\cX$ can be extended locally as a
positive, closed current with at most analytic singularities. Together
with Proposition~\ref{pr:extcurr} we have the following statement.

\begin{theorem}\label{th:singext}
Let $(\cK_{\cX/\cH},h)$ be the relative canonical bundle on the total
space over the Hilbert scheme, where the hermitian metric $h$ is induced
by the \ke metrics on the fibers. Then the curvature form extends to the
total space $\ol\cX$ over the compact Hilbert scheme $\ol\cH$ as a
positive, closed current $\omega^{KE}_{\ol \cX}$ with at most analytic
singularities.
\end{theorem}

Stronger estimates are to be expected for slices through boundary points
of the Hilbert scheme that correspond to log-canonical singularities.

\section{Extension of the \wp form for canonically polarized varieties
to the compactified Hilbert scheme}\label{se:compHilb} Following
\cite{f-s:extremal} for any (smooth) family $f:\cZ \to S$ of canonically
polarized manifolds, the generalized \wp form $\omega^{WP}_S$ on $S$ was
proven to be equal up to a numerical factor to a certain fiber integral.
We will use the notion $\simeq$ for equality up to a numerical factor.

We set $\cE=(\cK_{\cZ/S}- \cK_{\cZ/S}^{-1})^{n+1}$ in \eqref{eq:bgs}.
Equation \eqref{eq:fib0} in Section~\ref{se:bgs} yields
\begin{equation}\label{eq:fibint}
c_1(\lambda(\cE),h^Q) \simeq \int_{\cZ/S} \omega_\cZ^{n+1},
\end{equation}
where $\omega_\cZ= c_1(\cK_{\cZ/S},h)$.

On the other hand Lemma~\ref{le:varphi} together with
Proposition~\ref{pr:elleq} implied that the fiber integral equals the
generalized \wp form:
$$
\omega^{WP}(s) = \int_{\cZ_s} A^\alpha_{i\ol\beta} A^\ol\delta_{\ol\jmath\gamma}
g_{\alpha\ol\delta}g^{\ol\beta\gamma} g dV \ii ds^i\wedge ds^\ol\jmath,
$$
where the forms
$$
 A^\alpha_{i\ol\beta}\pt_\alpha dz^{\ol\beta}
$$
are the harmonic representatives of the \ks classes $\rho_s({\pt/\pt
s_i}|_s)$, i.e.\
\begin{equation}\label{eq:fibint_ke}
\omega^{WP}_S\simeq \int_{\cZ/S} c_1({\cK_{\cZ/S}}, h)^{n+1}.
\end{equation}
Now
\begin{equation}\label{eq:detbdl} c_1(\det f_!
((\cK_{\cZ/S}-\cK_{\cZ/S}^{-1})^{n+1}), h^Q) \simeq \omega^{WP}_S.
\footnote{Also the element $(\cK_{\cZ/S}-\cO_\cZ)^{n+1}$ of the
relative Grothendieck group can be taken instead.}
\end{equation}

We consider the situation of Hilbert schemes of canonically polarized
varieties.

After fixing the Hilbert polynomial and a multiple $m$ of the canonical
bundles, which yields very ampleness, we consider the universal embedded
family over the Hilbert scheme
$$
\xymatrix{\cX \ar[r]^i \ar[dr]_f & \mathbb P_N \times \cH \ar[d]^{pr} \\ & \cH.
}
$$
In this sense we modify the determinant line bundle and consider
$$
\lambda=\det f_! ((\cK_{\cX/S}^{\otimes m}-(\cK_{\cX/S}^{-1})^{\otimes m})^{n+1}),
$$
which only yields an extra factor $m^{n+1}$ in front of the \wp form.

Now we look at the compact Hilbert scheme $\ol\cH$ including points with
singular fibers
$$
\xymatrix{\ol\cX \ar[r]^{\ol i} \ar[dr]_{\ol f} & \mathbb P_N \times \ol \cH \ar[d]^{pr}
\\ & \ol\cH
}
$$
such that $\ol f$ is a flat morphism. Again we let $\cL$ be the pull-back
of the hyperplane section bundle to $\ol\cX$. Now the determinant line
bundle $\lambda$ extends to
$$
\ol\lambda=\det \ol f_! (\cL-(\cL^{-1})^{n+1})
$$
on $\ol\cH$.

We will need the fact that the construction of the Quillen metric of
determinant line bundles can be extended to smooth proper families over
base spaces with arbitrary singularities in such a way that $-\log h^Q$ is
locally the restriction of a $\cinf$ function given on a smooth ambient
space \cite[\S 12]{f-s:extremal}. This $\cinf$ function is a
$\pt\ol\pt$-potential of a real $(1,1)$-form, which restricts to the
generalized \wp form.

Next, we want to apply Theorem~\ref{th:singext} and consider the fiber
integral analogous to \eqref{eq:fibint_ke} for the map $\wt f: \wt\cX \to
\wt \cH$.
\begin{equation}\label{eq:extwp}
\omega_{\ol\cH}^{WP}=\int_{\ol\cX/\ol\cH} (\omega^{KE}_\ol\cX)^{n+1}.
\end{equation}
We need to show that the above fiber integral is well-defined.

\begin{theorem}\label{pr:wp_pot}
The extended \wp form on $\ol \cH$ given by the fiber integral
\eqref{eq:extwp} is a positive $(1,1)$-current.
\end{theorem}
\begin{proof}
Again, the statement is about the pull-back to the normalization, so that
we may assume normality.

We use the notation from Section~\ref{se:degfam} accordingly. We need to
consider the fiber integral \eqref{eq:extwp} over the smooth locus $\cH$
of the family, i.e.\ \eqref{eq:fibint_ke}, which is strictly positive in
effective directions.

We first observe that we can see from the relative version of the
Monge-Ampère equation \eqref{eq:ke1}, how the  Kähler form on the total
space, i.e.\ the curvature form of the relative canonical bundle (induced
by the \ke metrics on the fibers) differs from the restriction of the \fs
form on the total space.
\begin{gather*}
\omega_\cX^{KE}:=2\pi c_1(\cK_{\cX/\cH},h) =
   \ddb(\log e^u\Omega^0_{\cX})= \omega^0_\cX + \ddb u,
\end{gather*}
where $\omega^0_\cX$ stands for the \fs form pulled back to $\cX$.

We observe that
$$
(\omega^0_\cX +\ddb u)^{n+1} = (\omega^0_\cX)^{n+1} + \sum_{j=0}^n\ddb u
(\omega^0_\cX)^j(\omega^0_\cX + \ddb u)^{n-j}.
$$
Concerning the first term, by Proposition~\ref{pr:Hilbquil} the fiber
integral of $(\omega^0_\cX)^{n+1}$ possesses an extension as a positive,
closed current. For the second term we use the results of
Section~\ref{se:degfam}: Near any point of $\ol\cH\backslash \cH$ the
potentials $u$ are bounded from above uniformly with respect to the
parameter of the base (here we applied the correcting term
$|\sigma(s)|^{2(k+\ell)}$ from the proof of Theorem~\ref{th:singext}). In
order to avoid taking wedge products of currents we consider
\begin{gather*}
\int_{\cX/\cH} \ddb \Big( u \cdot(\omega^0_\cX)^j(\omega^0_\cX + \ddb u)^{n-j}\Big)= \\
\qquad \ddb \left(\int_{\cX/\cH} u\cdot (\omega^0_\cX)^j(\omega^0_\cX +
\ddb u)^{n-j}\right).
\end{gather*}
Now the $\pt\ol\pt$-potential is given by integrals over the fibers:
$$
\int_{\cX_s} u\cdot (\omega^0_\cX)^j(\omega^0_\cX +
\ddb u)^{n-j}
$$
The functions $u$ are known to be uniformly bounded from above, whereas
the integrals
$$
\int_{\cX_s} (\omega^0_\cX)^j(\omega^0_\cX +
\ddb u)^{n-j}
$$
are constant in $s$. This shows the boundedness from above of the
potential for the singular \wp metrics for families of \ke manifolds of
negative curvature. Finally, we invoke again Proposition~\ref{pr:extcurr}.
\end{proof}

\section{Moduli of canonically polarized manifolds}
In this section we give a short analytic/differential geometric proof of
the quasi-projectivity of moduli spaces of canonically polarized manifolds
depending upon the variation of the \ke metrics on such manifolds.

We begin with the fact that (after fixing the necessary numerical
invariants) we are looking at a component $\cM$ of the moduli space of
canonically polarized manifolds, endowed with a compactification $\ol\cM$,
which is an algebraic space.

The proof of the quasi-projectivity will be a consequence of the previous
sections:

On the moduli space $\cM$ of canonically polarized manifolds we are given
a certain determinant line bundle, equipped with a strictly positive
Quillen metric of class $\cinf$ in the orbifold sense (cf.\
\cite{f-s:extremal}), which is induced by the \ke volume forms. (It was
essential to have a functorial construction for holomorphic families for
both the line bundle $\lambda$ and the Quillen metric $h^Q$ in order to
descend to the moduli space.)

\begin{proposition}\label{pr:extwp}
The \wp\ form $\omega^{WP}$ on $\cM$ possesses an extension to $\ol\cM$ as
a positive, closed $(1,1)$-current with at most analytic singularities.
\end{proposition}
\begin{proof}
Let $\ol\cM \backslash \cM=\cD$ be the compactifying divisor. We
desingularize the pair $(\ol \cM,\cD)$. Let $\tau:\wt{\ol\cM} \to \ol\cM$
be the corresponding modification, and let $\wt \cD$ be the simple normal
crossings divisor in $\wt{\ol\cM}$ corresponding to $\cD$.

We know from the previous section that $\omega^{WP}$ possesses local
extensions: On any irreducible component of $\wt\cD$, we can chose a point
and find a neighborhood in $\wt{\ol\cM}$ so that $\tau^*\omega^{WP}$ can
be extended into this neighborhood by Theorem~\ref{pr:wp_pot}. Now we
apply Proposition~\ref{pr:extcurr}. This provides us with an extension of
$\tau^*\omega^{WP}$ to $\wt{\ol\cM}$. Its push forward under the proper
map $\tau$ yields the claim. (Observe that the property of having at most
analytic singularities is preserved by this operation, in contrast of
having vanishing Lelong numbers).
\end{proof}

As already mentioned any statement about positive currents on reduced
complex spaces is actually a statement that holds on the  normalization.

However, given the structure of a compactification of a moduli space, the
Hilbert scheme approach yielded \underline{\em local}  extensions of the
determinant line bundles (or a finite power resp.) to a (normal)
compactification of the normalized moduli space so that according to
Theorem~\ref{th:extlinebdl} the pull-back of the line bundle to the
normalization extends as a holomorphic line bundle with a positive,
singular, hermitian metric.

Since the positive hermitian line bundle over the moduli space itself is
not to be changed, we may have to modify the compactification in order to
ensure that the original line bundle extends -- this is done in
Proposition~\ref{pr:nonnormal}. Now the criterion \cite[Theorem 6]{s-t} is
applicable.

Altogether we proved the following fact, which yields the
quasi-projectivity of the moduli space \cite{v,viebuch,v2,ko}.

\break

\begin{theorem}\label{th:posbdlmod}
\strut
\begin{itemize}
  \item[(i)]Let $\cM$ be a component of the moduli space of
      canonically polarized manifolds. Then a tensor power of the
      determinant line bundle $\det f_! (\cK_{\cX/S}^{\otimes m} -
      (\cK_{\cX/S}^{\otimes m})^{-1})^{n+1}$ for holomorphic families
      together with the Quillen metric descends to the strictly
      positive hermitian line bundle $(\lambda, h^Q)$ on the moduli
      space $\cM$ in the orbifold sense.
  \item[(ii)] The curvature form of the determinant line bundle
      extends as a (semi-)positive current with at most analytic
      singularities to a certain compactification $\ol\cM$.
  \item[(ii)] A tensor power of the determinant line bundle extends as
      a line bundle $\wh \lambda$ to a compactification of the moduli
      space, the Quillen metric gives rise to a singular hermitian
      metric on $\wh\lambda$, which is strictly positive and of class
      $\cinf$ in the orbifold sense over the interior and has at most
      analytic singularities.
\end{itemize}
\end{theorem}

\section{Further applications}
We consider the direct image of $\cK_{\cX/S}^{\otimes 2}$ for families of
canonically polarized manifolds.

\begin{theorem}\label{th:appl}
 Let $S$ be a complex manifold and $f: \cX \to S$ a holomorphic family of
canonically polarized manifolds, equipped with \ke metrics of constant
negative curvature. Let the locally free sheaf
$$
f_* (\cK_{\cX/S}^{\otimes 2})
$$
be equipped with the induced $L^2$ metric. Then
\begin{itemize}
  \item[(i)] the sheaf $f_* (\cK_{\cX/S}^{\otimes 2})$ is
      semi-positive in the sense of Nakano, if $S$ is K\"ahler.
  \item[(ii)] the sheaf $f_* (\cK_{\cX/S}^{\otimes 2})$ is positive in
      the sense of Nakano, if the family is effectively parameterized
      everywhere.
\end{itemize}
\end{theorem}

The proof is an immediate consequence of our main theorem and a theorem of
Berndtsson \cite{berndtsson}.

\subsection{The classical \wp metric on \tei space}
It is known from the results of Wolpert that the classical \wp metric for
families of Riemann surfaces of genus larger than one has negative
curvature: According to \cite{wo} the sectional curvature is negative, and
the holomorphic sectional curvature is bounded from above by a negative
constant. A stronger curvature property, which is related to strong
rigidity, was shown in \cite{sch:teich}. The strongest result on curvature
by Liu, Sun, and Yau now follows immediately from Theorem~\ref{th:main}.
\begin{corollary}[\cite{lsy}]
The \wp metric on the \tei space of Riemann surfaces of genus $p>1$ is
dual Nakano negative.
\end{corollary}
\begin{proof}
Observe that for a universal family $f:\cX\to S$ the classical \wp metric
on $R^1f_* \cT_{\cX/S}$ corresponds to the $L^2$ metric on its dual bundle
$f_*(\cK_{\cX/S}^{\otimes2})$, which is Nakano positive according to
Theorem~\ref{th:appl}.
\end{proof}

\subsection{Curvature of the generalized \wp metric and related metrics}
In this section we present a new approach to questions related
hyperbolicity properties of moduli spaces and the existence of
non-isotrivial families in the sense of the hyperbolicity conjecture of
Shafarevich. We include immediate corollaries to our main theorem whiich
are closely related to known cases (\ e.g.\ \cite{b-v, keko, kv1, kv2, m,
v-z, v-z2}).

We pick up the notations from Section~1 (in case of a smooth base space
$S$ of arbitrary dimension. Let $f:\cX \to S$ be a smooth, proper
holomorphic map, whose fibers $\cX_s$, $s\in S$ are canonically polarized
varieties of dimension $n$, equipped with \ke metrics of constant Ricci
curvature equal to one. Let
$$
\rho_{s_0} : T_{s_0}S \to H^1(\cX_{s_0},\cT_{\cX_{s_0}})
$$
be the \ks map for a point $s\in S$. The induced $L^2$-metric on the space
of infinitesimal deformations is given by integration of the harmonic
representatives of the \ks classes of tangent vectors. These were
discussed in Section~\ref{se:posi}.

Explicitly the \wp hermitian inner product is defined as follows: Let
$(s^1,\ldots,s^k)$ be local holomorphic coordinates on $S$ such that the
given base point corresponds to the origin, and let
$(z,s)=(z^1,\ldots,z^n,s^1,\ldots,s^k)$ be local holomorphic coordinates
on $\cX$ with\\ $f(z,s)=s$.

Let
$$
\rho_s\left(\frac{\partial}{\partial s_i}\right) =
[A_{i\ol \beta}^\alpha \frac{\partial}{\partial z^\alpha} dz^\ol\beta] \in
 H^1(\cX_{s_0}, \cT_{\cX_{s_0}})
$$
with harmonic representative $A_{i\ol\beta}^\alpha$. Then (with the above
notations for the \ka manifold $X=\cX_{s_0}$
\begin{eqnarray}
   A^\alpha_{i\ol\beta;\ol\delta}&=&A^\alpha_{i\ol\delta;\ol\beta}\label{eq:clsd}\\
   0&=&g^{\ol\delta\gamma} A^\alpha_{i\ol\delta;\gamma}\label{eq:harm}\\
  A_{i\ol\beta\ol\delta}&=&A_{i\ol\delta\ol\beta}\label{eq:symm}.
\end{eqnarray}
The above equation \eqref{eq:clsd} is the $\db$-closedness,
\eqref{eq:harm} the harmonicity, and \eqref{eq:symm} reflects the close
relationship with the metric tensor.

\begin{theorem}\label{th:fins}
Any  compact subspace or relatively compact subset of the moduli space of
canonically polarized complex surfaces possesses a complex Finsler
orbifold metric, whose holomorphic curvature is bounded by a negative
constant.
\end{theorem}

In particular, the theorem implies that there exist no non-isotrivial
holomorphic families of canonically polarized complex surfaces over the
projective line or an elliptic curve.

We will use the fact that the holomorphic curvature of a Finsler metric at
a certain point $p$ in the direction of a tangent vector $v$ is the
supremum of the curvatures of the pull-back of the given Finsler metric to
a holomorphic disk through $p$ and tangent to $v$ (cf.\
\cite{abate-patrizio}). (For a hermitian metric, the holomorphic curvature
is known to be equal to the holomorphic sectional curvature).

These facts readily generalize to metrics of orbifold type.

The construction of the Finsler metric is by modifying the generalized \wp
metric. We recall the formula for its curvature denoting by $\Box$ the
complex Laplacian on functions and tensors resp. The functions $A_i\cdot
A_\ol\jmath$ are pointwise inner products of \ks tensors, whereas
$A_i\wedge A_\ol\jmath$ denotes a $(0,2)$-form with values in
$\Lambda^2\cT_{\cX_s}$ (cf. \cite{sch:curv}).

\begin{theorem}[\cite{sch:curv}]\label{th:wpcurv}
Let $f:\cX \to S$ be a local universal family of canonically polarized
manifolds with smooth base space $S$. Then the curvature tensor of the
generalized \wp metric equals
\begin{eqnarray}\label{eq:curvwp}
R_{i\ol\jmath k\ol\ell}^{\rm P W}(s)
&=&-\int_{\cX_s}(\Box+1)^{-1}(A_i\cdot A_\ol\jmath)(A_k\cdot A_\ol\ell)g\,d V\cr
&&-\int_{\cX_s}(\Box+1)^{-1}(A_i\cdot A_\ol\ell)(A_k\cdot A_\ol\jmath)g\,d V\cr
&&-\int_{\cX_s}(\Box -1)^{-1}(A_i\wedge A_k)\cdot(A_\ol\jmath\wedge A_\ol\ell)g\,d V.
\end{eqnarray}
\end{theorem}
For any harmonic \ks tensor $A=A^\alpha_{\ol\beta}\frac{\partial}{\partial
z^\alpha}dz^\ol\beta$ we denote by $H(A\wedge A)$ the harmonic part of
$A\wedge A$.

The theorem implies the following estimate:

\begin{corollary}\label{cor:curv}
Let $A=\xi^iA_i$. Then
\begin{equation}\label{eq:corcurv}
R_{i\ol\jmath k\ol\ell}^{\rm P W}(s)\xi^i\xi^\ol\jmath\xi^k\xi^\ol\ell\leq
(-2\|A\|^4_{WP}+ \|H(A\wedge A)\|^2)/{\rm vol}(\cX_s).
\end{equation}
\end{corollary}
\begin{proof}
We apply the eigenspace decompositions of the function $A\cdot\ol A$ and
the tensor $A\wedge A$ with respect to the Laplacians. It was shown in
\cite{sch:curv} that the eigenspace decomposition of $A\wedge A$ contains
no contributions for eigenvalues $\lambda\in (0,1]$.
\end{proof}

Now we denote by $G$ the Finsler metric induced by $G^{WP}_{i\ol \jmath}$.
It is known that the holomorphic curvature of $G$ is equal to the
holomorphic sectional curvature of $G^{WP}_{i\ol \jmath}$ (cf.\
\cite{abate-patrizio}).

From now on, we assume that the fibers are of {\it complex dimension two}.
The locally free sheaf $R^2f_*\Lambda^2  \cT_{\cX/S}$ is dual to
$f_*\cK^{\otimes2}_{\cX/S}$. The latter, equipped with the induced
$L^2$-inner product, is Nakano-positive according to Theorem~\ref{th:appl}
for any effectively parameterized family $f:\cX \to S$. However, at this
point, we cannot give any estimate for the curvature because
Theorem~\ref{th:main} does not contain any estimates.

We consider the natural morphism
\begin{equation}
\mu:S^2 T_S \to R^2\!f_*\Lambda^2\cT_{\cX/S}.
\end{equation}
In general, we can only say that it induces a Finsler semi-metric on $S$.
If the semi-metric is not identically zero but vanishes only on a thin
analytic subset, it is of non-positive holomorphic curvature (considering
that the holomorphic curvature of a Finsler metric is defined in terms of
holomorphic curves).

We need the following fact. Let $C$ be a holomorphic curve and
$G=G(z)dz\ol{dz}$ a hermitian semi-metric, which is positive on the
complement of a discrete subset say. Denote by
$$
K_G:= - \left.\frac{\pt^2\log G(z)}{\pt z \ol{\pt z}}\!\!\right/\!\! G(z)
$$
the (Ricci) curvature. Let  $H=H(z)dz\ol{dz}$ be a further such metric.
\begin{lemma}[cf.\ {\cite[Lemma~3]{sch:framas}}]\label{le:convsum}
\begin{equation}\label{eq:curvest}
K_{G+H} \leq \frac{G^2}{(G+H)^2} K_G +  \frac{H^2}{(G+H)^2} K_H.
\end{equation}
\end{lemma}
Observe that for $H\equiv 0$ the equation \eqref{eq:curvest} formally
still holds.

Let again $f:\cX \to S$ be a local, universal family of canonically
polarized manifolds. Then the \wp metric determines a Finsler metric $G$
on $S$, and the dual Nakano negative bundle $R^2f_*\Lambda^2\cT_{\cX/S}$
determines a Hermitian semi-metric $H\not \equiv 0$ for every curve $C$
with $\mu|_C$ not identically zero, since the map $\mu$ restricted to $C$
maps $\cT_C^{\otimes 2}$ to the hermitian bundle
$R^2f_*\Lambda^2\cT_{\cX/S}|C$ (compatible with base change).

Now, we can use Corollary~\ref{cor:curv} and Lemma~\ref{le:convsum} (under
the assumption on the base space in Theorem~\ref{th:fins}) to construct
the desired Finsler orbifold metric from a convex sum $G+\gamma H$ of $G$
and $H$, whose curvature is bounded by a negative constant from above.

The non-existence of non-isotrivial holomorphic families of canonically
polarized surfaces over compact curves $C$ of genus zero or one can be
seen directly from Theorem~\ref{th:wpcurv} and Theorem~\ref{th:main}. If
the map $\mu$ on $C$ is identically zero, then the curvature formula for
the \wp metric \eqref{eq:curvwp} and the estimate \eqref{eq:corcurv} imply
the claim, if not, Theorem~\ref{th:appl} can be applied directly.

\begin{corollary}[]\label{th:arty}
Let $C$ be a smooth, compact curve and $f:\cX \to C$ a non-isotrivial
family of canonically polarized surfaces. Then $g(C)>1$ or there exists at
least one singular fiber.
\end{corollary}
\begin{proof}
We apply Theorem~\ref{th:wpcurv} and Theorem~\ref{th:main}: Let $C'\subset
C$ be the set of points with regular fibers. Consider the case, where the
map $\mu$ on $C'\subset C$ is identically zero. Then the curvature formula
for the \wp metric \eqref{eq:curvwp} and the estimate \eqref{eq:corcurv}
imply the existence of a metric, whose curvature is bounded from above by
a negative constant. If $\mu\not\equiv0$ we apply the Gauß-Bonnet theorem
for singular metrics unsing Theorem~\ref{th:singext}. Since in this
situation, we do not have a negative upper estimate for the curvature, we
cannot bound the Lelong numbers at the singularities. So we can only infer
the existence of at least one singular fiber, if $g(C)\leq 1$.
\end{proof}


\begin{thebibliography}{MM-M}

\bibitem[A-P]{abate-patrizio} Abate, M., Patrizio, G.: Holomorphic
    curvature
    of Finsler metrics and complex geodesics. J.\ Geom.\ Anal.\ {\bf 6}, 341--363
    (1996).

\bibitem[AU]{aub} Aubin, T.: Equation du type de Monge-Ampère sur les
    variétés Kähleriennes compactes, C. R. Acad. Sci. Prais {\bf 283}, 119--121(1976) /
    Bull.\ Sci.\ Math.\ {\bf 102}, 63--95 (1978).

\bibitem[A]{aust} Aust, H.: A criterion for the quasi-projectivity of
    complex spaces, forthcoming thesis, Marburg.

\bibitem[A-S]{a-s} Axelsson, R., Schumacher, G.: Kähler geometry of
    Douady spaces. Manuscr.\ math. {\bf 121}, 277--291 (2006).


\bibitem[B-V]{b-v} Bedulev E., Viehweg, E.:Shafarevich conjecture for
    surfaces of general type over function fields. Invent.\ math.\ {\bf 139},
    603--615 (2000).

\bibitem[B]{berndtsson} Berndtsson, B.: Curvature of vector bundles
    associated
    to holomorphic fibrations. arXiv:math/0511225v2 [math.CV] 20 Aug 2007.

\bibitem[BI]{bi1} Bismut, J.-M.: M\'etriques de Quillen et
    d\'eg\'en\'erescence de vari\'et\'es kählériennes. C.\ R.\ Acad.\
    Sci.\ Paris Sér.\ I.\ Math.\ 319 (1994), 1287--1291.

\bibitem[BGS]{bgs} Bismut, J.-M.; Gillet, H.; Soul\'e, Ch.: Analytic
    torsion and holomorphic determinant bundles I, II, III. Comm.\ Math.\
    Phys.\ 115 (1988), 49--78, 79--126, 301--351.

\bibitem[B-S]{b-s} Biswas, I.; Schumacher, G.: Generalized Petersson-Weil
    metric on the Douady space of embedded manifolds. Complex analysis and
    algebraic geometry, 109--115, de Gruyter, Berlin, 2000.

\bibitem[BO]{bom} Bombieri, E.: Algebraic values of meromorphic maps,
    Invent.\ math.\ {\bf 10} (1970) 267--287 and Addendum Invent.\ math.\
    {\bf 11} 163--166 (1970).

\bibitem[C-Y]{cheng-yau} Cheng, S.Y., Yau, S.T.: On the existence of a
    complete K\"ahler metric on noncompact complex manifolds and the regularity of
    Fefferman's equation. Comm.\ Pure Appl.\ Math. {\bf 33}, 507--544 (1980).

\bibitem[DT-K]{kdt} DeTurck, D., Kazdan, J.: Some regularity theorems in
     Riemannian geometry.  Ann.\ Sci.\ Éc.\ Norm.\ Supér.\ {\bf 14} 249--260
    (1981).

\bibitem[EM]{em} El Mir, H.: Sur le prolongement des courants positifs
    fermées, Acta Math. {\bf 153}, 1--45, (1984).

\bibitem[FO]{fo} Forster, O.: Zur Theorie der Steinschen Algebren und
    Moduln. Math.\ Z.\ {\bf 97}, 376--405 (1967).

\bibitem[F-S]{f-s:extremal} Fujiki, A., Schumacher, G.: The moduli space
    of
    extremal compact \ka  manifolds and generalized Weil-Petersson metrics.
    Publ.\ Res.\ Inst.\ Math.\ Sci.\ {\bf 26}, 101--183 (1990).

\bibitem[HOU]{hou} Houzel, Ch.: Géometrie analytique locale, II. Théorie
    des morphismes finis. Séminaire Cartan, 13e année, 1969/61, no {\bf 19}.

\bibitem[G-R]{g-r} Grauert, H., Remmert, R.: Plurisubharmonische
    Funktionen in komplexen R\"aumen.  Math.\ Z.\ {\bf 65} 175--194 (1956).

\bibitem[KA]{kawa} Kawamata, Y.: Characterization of Abelian Varieties.
    Compos.\ Math. {\bf 43} 253--276 (1981).

\bibitem[KE-KO]{keko} Kebekus, S., Kovács, S.: Families of canonically
    polarized varieties over surfaces.  Invent.\ Math.\ {\bf 172},
    657--682 (2008).

\bibitem[K]{kob} Kobayashi, R.: K\"ahler-Einstein metric on an open
    algebraic
    manifold. Osaka J.\ Math. {\bf 21},  399--418 (1984).

\bibitem[KO]{ko} Koll\'{a}r, J.: Projectivity of complete moduli, J.\ of
    Diff.\ Geom.\ 32 (1990), 235--268.

\bibitem[KO1]{ko1} Koll\'{a}r, J.: Non-quasi-projective moduli spaces.
    Ann.\  Math.\ {\bf 164} (2006), 1077--1096.

\bibitem[KO-KO]{kk} Kollár, J., Kovács, S.: Log canonical singularities
    are Du Bois. arXiv:0902.0648.

\bibitem[K-M]{kolmo} Kollár, J., Mori, S.: Birational Geometry of
    Algebraic Varieties. Cambridge University Press 1998. Translated from
    Souyuuri Kikagaku published by Iwanami Shoten, Publishers, Tokyo,
    1998.

\bibitem[KV1]{kv1} Kovács, S.: Smooth families over rational and elliptic
    curves. J.\ Alg.\ Geom.\ {\bf 5} 369--385 (1996).

\bibitem[KV2]{kv2} Kovács, S.: On the minimal number of singular fibres in
    a family of surfaces of general type. Journ.\ Reine Angew.\ Math.\ {\bf
    487}, 171--177 (1997).

\bibitem[KV3]{kv3} Kovács, Algebraic hyperbolicity of fine moduli
    spaces, J.\ Alg.\ Geom.\ {\bf 9} (2000) 165--174.

\bibitem[M]{m} Migliorini, L.: A smooth family of minimal surfaces of
    general type over a curve of genus at most one is trivial. J.\ Alg.\ Geom.\
    {\bf 4}, 353--361 (1995).

\bibitem[M-T]{mouta} Mourougane, Ch., Takayama, S.: Hodge metrics and the
    curvature of higher direct images. arXiv:0707.3551v1 [math.AG] 24 Jul 2007

\bibitem[L-S-Y]{lsy} Liu, K., Sun, X., Yau, S.T.: Good geometry and moduli
    spaces, I. To appear.

\bibitem[P-W]{pw}
    Protter, M.H.; Weinberger, H.F.: Maximum-principles in different equations.
    Englewood Cliffs, N.J.: Prentice-Hall, Inc., (1967).

\bibitem[SCH1]{sch:teich} Schumacher, G.: Harmonic maps of the moduli
    space of compact Riemann surfaces.  Math.\ Ann.\ {\bf 275}, 455--466 (1986).

\bibitem[SCH2]{sch:curv} Schumacher, G.: The curvature of the
    Petersson-Weil metric on the moduli space of Kähler-Einstein manifolds.
    Ancona, V.\ (ed.) et al., Complex analysis and geometry. New York:
    Plenum Press. The University Series in Mathematics. 339--354 (1993).

\bibitem[SCH3]{sch:framas} Schumacher, G.: Asymptotics of \ke metrics on
    quasi-projective manifolds and an extension theorem on
    holomorphic maps. Math.\ Ann.\ {\bf 311}, 631--645 (1998)

\bibitem[S-T]{s-t} Schumacher, G., Tsuji, H.: Quasi-projectivity of moduli
    spaces of polarized varieties. Ann. Math. {\bf 159}, 597--639 (2004).

\bibitem[SI]{sib} Sibony, N.: Quelques problèmes de prolongement de
    courants en analyse complexe. Duke Math.\ J.\ {\bf 52}, 157--197 (1985).

\bibitem[SIU1]{siu:gap} Siu, Y.T.: Absolute gap-sheaves and extensions of
    coherent analytic sheaves. Trans.\ Am.\ Math.\ Soc.\ {\bf 141}, 361--376 (1969).

\bibitem[SIU2]{siu:curr} Siu, Y.-T. Analyticity of sets associated to
    Lelong numbers and the extension of closed positive currents.  Invent.\
    Math.\ {\bf 27}, 53--156 (1974).

\bibitem[SIU3]{siu:canlift} Siu, Y.-T.: Curvature of the Weil-Petersson
    metric in the moduli space of compact Kähler-Einstein manifolds of negative first
    Chern class. Contributions to several complex variables, Hon.\ W.\ Stoll,
    Proc.\ Conf.\ Complex Analysis, Notre Dame/Indiana 1984, Aspects Math.\ E9,
    261--298 (1986).

\bibitem[SIU4]{siu:pluri} Siu, Y.-T.: Invariance of plurigenera. Invent.\
    Math.\  {\bf 134},  661--673 (1998).

\bibitem[SK]{sk} Skoda,H.: Sous-ensembles analytiques d'ordre fini ou
    infini dans $\C^n$. Bull.\ Soc.\ Math.\ France {\bf 100}, 353--408 (1972).

\bibitem[SO]{so} Sommese, A.: Criteria for quasi-projectivity. Math.\
    Ann.\ {\bf 217} 247--256 (1975).

\bibitem[VA1]{va1} Varouchas, J.: Stabilit\'e de la classe des variet\'es
    K\"ahleriennes par certains morphismes propres. Invent.\  math.\ {\bf 77},
    117--127 (1984).

\bibitem[VA2]{va2} Varouchas, J.: K\"ahler spaces and proper open
    morphisms.  Math.\ Ann.\ {\bf 283}, 13--52 (1989).

\bibitem[V]{v} Viehweg, E.: Weak positivity and stability of certain
    Hilbert points I,II,III, Invent.\ Math.\ {\bf 96}, 639--669 (1989).
    Invent. Math. {\bf 101} 191--223 (1990), Invent. Math.{\bf 101}
    521--543 (1990).

\bibitem[V1]{viebuch} Viehweg, E.: Quasi-projective moduli for
    polarized manifolds. Ergebnisse der Mathematik und ihrer
    Grenzgebiete. 3.~Folge. {\bf 30}. Berlin: Springer-Verlag, 1995.

\bibitem[V2]{v2} Viehweg, E.: Compactifications of smooth families and of
    moduli spaces of polarized manifolds. arXiv:math/0605093v2

\bibitem[V-Z1]{v-z} Viehweg, E.; Zuo, K.: On the Brody hyperbolicity of
    moduli spaces for canonically polarized manifolds. Duke Math.\
    J.\ {\bf 118}, 103--150 (2003).

\bibitem[V-Z2]{v-z2} Viehweg, E.; Zuo, K.: On the isotriviality of
    families of projective manifolds over curves, J.\ Algebraic Geom.\
    {\bf 10}, 781--799 (2001).

\bibitem[WO]{wo} Wolpert, S.: Chern forms and the Riemann tensor
    for the moduli space of curves. Invent.\ Math.\ {\bf 85}, 119--145
    (1986).

\bibitem[Y]{yau} Yau, S.T.: On the Ricci curvature of a compact Kähler
    manifold and the complex Monge-Ampère equation, Commun.\ Pure Appl.\
     Math.\ {\bf 31}, 339--411 (1978).

\end{thebibliography}
\end{document}